\documentclass[a4paper]{article}
\usepackage[utf8]{inputenc}
\usepackage[margin=1.3in]{geometry}
\usepackage{amssymb}
\usepackage{amsmath}
\usepackage{amsthm}
\usepackage{booktabs}
\usepackage{mathrsfs}
\usepackage{verbatim}
\usepackage{lscape}
\usepackage{mathtools}
\usepackage{float}
\usepackage{enumitem}
\usepackage[justification=centering]{caption}
\usepackage{subcaption}
\usepackage{longtable}
\usepackage{xspace}
\usepackage[ruled]{algorithm2e}
\SetAlgoSkip{bigskip}
\SetKwInput{KwData}{Input}
\SetKwInput{KwResult}{Output}
\SetKwComment{Comment}{// }{}
\usepackage{tikz-cd}
\usepackage{thmtools}
\usepackage{cite}
\usepackage{hyperref}

\theoremstyle{plain}
\newtheorem{theorem}{Theorem}[section]

\newtheorem{proposition}[theorem]{Proposition}
\newtheorem{lemma}[theorem]{Lemma}
\newtheorem{example}[theorem]{Example}
\newtheorem{remark}[theorem]{Remark}
\newtheorem{definition}[theorem]{Definition}

\newcommand{\cechtext}{\v{C}ech\xspace}

\newcommand{\Z}{\mathbb{Z}}

\newcommand{\R}{\mathbb{R}}

\newcommand{\FF}{\mathcal{F}}

\newcommand{\II}{\mathcal{I}}
\newcommand{\RR}{\mathcal{R}}

\newcommand{\HH}{{\mathrm{H}}}
\newcommand{\PH}{{\mathrm{PH}}}
\newcommand{\D}{{\mathrm{D}}}
\newcommand{\DR}{{\mathrm{DR}}}

\newcommand{\vr}{\mathrm{VR}}

\newcommand{\cech}{\mathrm{\check{C}}}

\newcommand{\uu}[1]{\underline{x}}

\newcommand{\hurl}[1]{\href{https://#1}{\texttt{#1}}}
\newcommand{\email}[1]{\href{mailto:#1}{\texttt{#1}}}

\graphicspath{{figures/}}

\title{Flagifying the Dowker Complex}

\author{Marius Huber\thanks{Supported by the Swiss National Science Foundation (project no. 209413).} \\
        Department of Computational Linguistics,
        University of Zürich, Switzerland \\ {\email{marius.huber@uzh.ch}}
        \and
        Patrick Schnider \\
        Department of Mathematics and Computer Science,
        University of Basel\\
        Department of Computer Science, ETH Z\"{u}rich, Switzerland \\ {\email{patrick.schnider@inf.ethz.ch}}
}
\date{}

\begin{document}

\maketitle

\begin{abstract}
The Dowker complex \(\D_{R}(X,Y)\) is a simplicial complex capturing the topological interplay between two finite sets \(X\) and \(Y\) under some relation \(R\subseteq X\times Y\).
While its definition is asymmetric, the famous Dowker duality states that \(\D_{R}(X,Y)\) and \(\D_{R}(Y,X)\) have homotopy equivalent geometric realizations.
We introduce the Dowker-Rips complex \(\DR_{R}(X,Y)\), defined as the flagification of the Dowker complex or, equivalently, as the maximal simplicial complex whose \(1\)-skeleton coincides with that of \(\D_{R}(X,Y)\).
This is motivated by applications in topological data analysis, since as a flag complex, the Dowker-Rips complex is less expensive to compute than the Dowker complex.
While the Dowker duality does not hold for Dowker-Rips complexes in general, we show that one still has that \(\HH_{i}(\DR_{R}(X,Y))\cong\HH_{i}(\DR_{R}(Y,X))\) for \(i=0,1\).
We further show that this weakened duality extends to the setting of persistent homology, and quantify the ``failure" of the Dowker duality in homological dimensions higher than \(1\) by means of interleavings.
This makes the Dowker-Rips complex a less expensive, approximate version of the Dowker complex that is usable in topological data analysis.
Indeed, we provide a Python implementation of the Dowker-Rips complex and, as an application, we show that it can be used as a drop-in replacement for the Dowker complex in a tumor microenvironment classification pipeline.
In that pipeline, using the Dowker-Rips complex leads to increase in speed while retaining classification performance.
\end{abstract}


\section{Introduction}\label{sec:introduction}

Topological data analysis (TDA) provides a framework for extracting qualitative geometric and topological features from complex data sets.
Central to this approach is the construction of simplicial complexes that approximate the shape of an data set or, more generally, a metric space.
A prominent example of such a complex is the \cechtext complex, where a finite set of points is declared to span a simplex precisely if the balls of some fixed radius \(\varepsilon>0\) around the points have non-empty intersection.
While the \cechtext complex provably captures the topology of the union of all \(\varepsilon\)-balls, it is notoriously expensive to compute because triple and higher order intersections of balls must be checked (see e.g.~\cite[Chapter 2.5]{ghrist_elementary_applied_topology} and~\cite[Chapter III]{edelsbrunner_computational_topology}).
As a way around this, one often resorts to working with a simpler complex known as the Vietoris-Rips complex in practice.
By definition, the Vietoris-Rips complex is obtained by flagifying of the \cechtext complex, that is, by adding all possible simplices whose edges are already present in the \cechtext complex.
By construction, the Vietoris-Rips complex is thus entirely determined by its \(1\)-skeleton, which coincides with that of the \cechtext complex.
This makes the Vietoris-Rips complex less expensive to describe, compute and store.
Indeed, several software packages for computing persistent homology like GUDHI \cite{maria_gudhi} and Ripser \cite{bauer_ripser} allow for a significant speed-up in computation time when working with flag complexes.
Moreover, even though the Vietoris-Rips complex does not enjoy the same theoretical guarantees regarding the capturing of the topology of the underlying data set, it is guaranteed to be ``topologically close" to the \cechtext complex in the sense that the two complexes are interleaved.
Finally, there do exist conditions under which such guarantees for the Vietoris-Rips complex do exist~\cite{chambers_de_silva_vietoris_rips_of_planar_sets,attali_vietoris_rips_reconstruction_of_samples_shapes}.

While both the \cechtext and Vietoris-Rips complexes are used to analyze a single data set, one might be interested in analyzing the topology of a data set relative to another one living in the same space (or, equivalently, the topology of a subset of a data set relative to its complement).
One tool for doing so is the Dowker complex, which was introduced by Dowker in 1952~\cite{dowker_homology_groups_of_relations}.

\begin{definition}\label{def:dowker_complex}
    Let \(X,Y\) be two finite sets and let \(R\subseteq X\times Y\) be a non-empty relation.
    The \emph{Dowker complex on \(X\) relative to \(Y\)} is the simplicial complex \(\D_{R}(X,Y)\) defined by the rule that a finite subset \(\sigma\subseteq X\) belongs to \(\D_{R}(X,Y)\) iff there exists \(y\in Y\) such that \((x,y)\in R\) for all \(x\in\sigma\).
\end{definition}

If \(X\) and \(Y\) in Definition~\ref{def:dowker_complex} are subsets of a metric space \((Z,d)\), one may define a relation \(R_{\varepsilon}\subseteq X\times Y\) by declaring \((x,y)\in R_{\varepsilon}\) iff \(d(x,y)\leq\varepsilon\) for \(\varepsilon\geq 0\).
In this setting, the Dowker complex may be regarded as a variant of the \cechtext complex where one does not simply require the intersection of \(\varepsilon\)-balls around elements of \(X\) to be non-empty, but indeed to contain an element of \(Y\).

A particularly nice feature of the Dowker complex is given by the \emph{Dowker duality}, proven by Dowker in the original paper introducing Dowker complexes \cite{dowker_homology_groups_of_relations}.
It states that the two complexes \(\D_{R}(X,Y)\) and \(\D_{R}(Y,X)\) are homotopy equivalent and, in particular, have isomorphic homology groups.
This result has been extended to filtrations by Chowdhury and M\'{e}moli, who have shown that these homotopy equivalences commute with the inclusions of the filtrations, thus extending Dowker duality to the setting of persistent homology~\cite{chowdhury_memoli_functorial_dowker_theorem}.
In other words, this more general form of Dowker duality allows one to compute persistent homology for an entire filtration of complexes \(\left\{\D_{R}(X,Y)\right\}_{R\in\RR}\) for some set \(\RR\) of nested relations, and this persistent homology is guaranteed to coincide with that of the corresponding filtration \(\left\{\D_{R}(Y,X)\right\}_{R\in\RR}\).
In particular, this may be applied to the relations \(R_{\varepsilon}\) in the setting of metric spaces.
From a practical perspective, this duality allows one to compute the smaller of the two complexes at each step (which amounts to potentially swapping the roles of \(X\) and \(Y\)).
This can be crucial for computation time and memory consumption, in particular if one of \(X\) and \(Y\) is significantly smaller than the other.
In the context of metric spaces, the persistence diagrams resulting from filtrations of Dowker complexes provide a way of analyzing whether and how the classes \(X\) and \(Y\) are colocalized in the ambient metric space \(Z\) (see e.g.~\cite[Section 5.1.2]{stolz_relational_persistent_homology} for details).
Dowker complexes have seen applications inside math as well as outside of math, in domains as diverse as computational biology, data science, machine learning and neuroscience ~\cite{stolz_relational_persistent_homology,choi_revisiting_link_prediction,brun_blaser_sparse_dowker_nerves,zemene_clustering,liu_dowker_based_machine_learning,moshkov_predicting_compound_activity,vaupel_topological_perspective_on_dual_nature_of_neural_state_space,freund_lattice_base_and_topological_representations,garland_exploring_topology_dynamical_reconstructions}.
For more details on Dowker complexes, see, for instance,~\cite{chazal_stability_for_geometric_complexes,ghrist_elementary_applied_topology,chowdhury_memoli_functorial_dowker_theorem}.

In this work, we introduce and examine a flagified version of the Dowker complex, which we call the \emph{Dowker-Rips complex}.
Just as the Vietoris-Rips complex can be regarded as a less expensive, but approximate variant of the \cechtext complex, the Dowker-Rips complex can be regarded as such a variant of the Dowker complex.
To define the Dowker-Rips complex, we first state a precise definition of flagifications.

\begin{definition}\label{def:k_flagification}
    Given a simplicial complex \(X\), the \emph{flagification of \(X\)}, denoted by \(\FF(X)\), is defined as the simplicial complex that is obtained from \(X\) by including a simplex \(\sigma\subseteq X\) whenever all edges of \(\sigma\) already belong to \(X\) and \(\dim(\sigma)\geq 2\).
    More generally, for an integer \(k\geq 2\), the \emph{\(k\)-flagification of \(X\)}, denoted by \(\FF^{\geq k}(X)\), is defined as the complex that is obtained from \(X\) by including a simplex \(\sigma\subseteq X\) whenever all \((k-1)\)-dimensional faces of \(\sigma\) already belong to \(X\) and \(\dim(\sigma)\geq k\).
\end{definition}

\begin{remark}
    Note that \(X\subseteq\FF^{\geq k}(X)\subseteq\FF(X)\) for any simplicial complex \(X\) and \(k\geq 2\).
    Moreover, we have that \(X=\FF^{\geq k}(X)\) if \(k>\dim(X)+1\), and \(\FF^{\geq 2}(X)=\FF(X)\) for any simplicial complex \(X\).
\end{remark}

\begin{example}\label{exp:k_flagification}
    Let \((X,d)\) be a metric space, and denote by \(\cech_{\varepsilon}(X)\) and \(\vr_{\varepsilon}(X)\) its \cechtext and Vietoris-Rips complexes at some scale \({\varepsilon}\geq 0\), respectively.
    Then we have that \(\vr_{\varepsilon}(X)=\FF(\cech_{\varepsilon}(X))\).
\end{example}

With the definition of flagification at hand, we are now ready to define the Dowker-Rips complex.

\begin{definition}\label{def:dowker_rips_complex}
    Let \(X,Y\) be two finite sets and let \(R\subseteq X\times Y\) be a non-empty relation.
    The \emph{Dowker-Rips complex on \(X\) relative to \(Y\)} is defined as \[
        \DR_{R}(X,Y)\coloneqq\FF(\D_{R}(X,Y)).
    \]
\end{definition}

There are two natural questions that arise:
\begin{enumerate}[label=(\arabic*), ref=(\arabic*)]
    \item\label{question:difference_dr_vs_d} How much can the Dowker-Rips complex differ from the Dowker complex?
    \item\label{question:existence_duality} Does some version of the Dowker duality still hold for Dowker-Rips complexes?
\end{enumerate}

For filtrations of simplicial complexes, questions such as Question~\ref{question:difference_dr_vs_d} are usually answered by showing that the two filtrations are \emph{multiplicatively \(c\)-interleaved} for some \(c\geq 1\).\footnote{
    For the definition of a multiplicative interleaving, see Definition~\ref{def:mult_interleaving}.
}
Informally speaking, the smaller the value of \(c\geq 1\), the closer the two filtrations are.
A prominent example of this is the chain of inclusions
\begin{equation}\label{eq:interleaving_vr_cech}
    \cech_{\varepsilon}(X)\subseteq\vr_{\varepsilon}(X)\subseteq\cech_{2\varepsilon}(X)
\end{equation}
for \(\varepsilon\geq 0\), which translates into the fact that the Vietoris-Rips complex and the \cechtext complex are multiplicatively \(2\)-interleaved.
We show that a similar argument also works for Dowker-Rips and Dowker complexes in the case where \(X\) and \(Y\) are subsets of some metric space \((Z,d)\) with the relation \(R_{\varepsilon}\subseteq X\times Y\) defined by declaring \((x,y)\in R\) iff \(d(x,y)\leq\varepsilon\) for \(\varepsilon\geq 0\).

\begin{restatable}{theorem}{DowkerInterleaving}\label{thm:interleaving_with_dowker}
    Let \(X,Y\subseteq Z\) where \((Z,d)\) is some metric space, and define the relations \(R_{\varepsilon}\subseteq X\times Y\) by declaring \((x,y)\in R\) iff \(d(x,y)\leq\varepsilon\) for \(\varepsilon\geq 0\).
    Denote by \(\D_{\bullet}(X,Y)\) the filtration given by \(\left\{\D_{R_{\varepsilon}}(X,Y)\right\}_{\varepsilon\in\R^{+}}\), and similarly for \(\DR_{\bullet}(X,Y)\).
    Then have that
    \begin{equation}\label{eq:interleaving_with_dowker}
        \D_{\varepsilon}(X,Y)\subseteq\DR_{\varepsilon}(X,Y)\subseteq\D_{3\varepsilon}(X,Y)
    \end{equation}
    for all \(\varepsilon\geq 0\), so that \(\D_{\bullet}(X,Y)\) and \(\DR_{\bullet}(X,Y)\) are multiplicatively $3$-interleaved.
\end{restatable}

The above result is sharp in the sense that the inclusion \(\DR_{\varepsilon}(X,Y)\subseteq\D_{3\varepsilon}(X,Y)\) does not hold when \(3\) is replaced by some value \(c<3\) (see Proposition~\ref{prop:counterexample_interleaving} for such an example).

We use a similar argument to give a partial answer to Question~\ref{question:existence_duality}.
We point out that the multiplicative interleaving claimed in the following does not stem from a chain of inclusions such as in~\eqref{eq:interleaving_vr_cech} and~\eqref{eq:interleaving_with_dowker}, but rather from the more general notion of a multiplicative interleaving defined in Section~\ref{sec:interleaving}.

\begin{restatable}{theorem}{FlipInterleaving}\label{thm:interleaving_with_flipped}
    Let \(X,Y\subseteq Z\) where \((Z,d)\) is some metric space, and define the relations \(R_{\varepsilon}\subseteq X\times Y\) as in Theorem~\ref{thm:interleaving_with_dowker}, \(\varepsilon\geq 0\).
    Denote by \(\DR_{\bullet}(X,Y)\) the filtration given by \(\left\{\DR_{R_{\varepsilon}}(X,Y)\right\}_{\varepsilon\in\R^{+}}\), and similarly for \(\DR_{\bullet}(Y,X)\).
    Then \(\DR_{\bullet}(X,Y)\) and \(\DR_{\bullet}(Y,X)\) are multiplicatively $3$-interleaved.
\end{restatable}

While this already establishes that \(\DR_{\bullet}(X,Y)\) and \(\DR_{\bullet}(Y,X)\) cannot be ``too different'', it is still a significantly weaker guarantee than the one we have for Dowker complexes, where we have a homotopy equivalence and thus an isomorphism at the level of persistent homology.
Indeed, as we will see in Section~\ref{sec:duality}, an isomorphism at the level of persistent homologies of \(\DR_{\bullet}(X,Y)\) and \(\DR_{\bullet}(Y,X)\) does not exist in general.
Nevertheless, we still obtain an isomorphism at the level of persistent homology when restricted to homological dimensions \(0\) and \(1\).
This follows from a slightly more general result on \(k\)-flagifications of Dowker complexes.

\begin{restatable}{theorem}{ThmMainTechnical}\label{thm:main_technical}
    Let \(X\) and \(Y\) be two finite sets and let \(\left\{R_{j}\right\}_{j\in J}\) be a sequence of relations such that \(R_{j}\subseteq X\times Y\) for all \(j\in J\), and \(R_{j}\subseteq R_{j'}\) whenever \(j\leq j'\), where \(J\) is some totally ordered index set.
    Given an integer \(k\geq 2\), denote by \(\FF^{\geq k}(\D_{\bullet}(X,Y))\) the filtration given by \(\left\{\FF^{\geq k}(\D_{R_{j}}(X,Y))\right\}_{j\in J}\), and similarly for \(\FF^{\geq k}(\D_{\bullet}(Y,X))\).
    Then we have that \[
        \PH_{i}(\FF^{\geq k}(\D_{\bullet}(X,Y)))\cong\PH_{i}(\FF^{\geq k}(\D_{\bullet}(Y,X)))
    \] for \(i=0,\dots,k-1\).
\end{restatable}

\begin{remark}
    For large enough \(k\geq 1\), we have that \(\FF^{\geq k}(\D_{R}(X,Y))=\D_{R}(X,Y)\) and \(\FF^{\geq k}(\D_{R}(Y,X))=\D_{R}(Y,X)\).
    For such choices of \(k\), Theorem~\ref{thm:main_technical} is essentially a homological (and hence weaker) restatement of~\cite[Theorem 3]{chowdhury_memoli_functorial_dowker_theorem}.
    Indeed, Theorem~\ref{thm:main_technical} may be read as saying that there exists a decreasing sequence of filtrations \[
        \DR_{\bullet}(X,Y)=\FF^{\geq 2}(\D_{\bullet}(X,Y))\supseteq\cdots\supseteq\FF^{\geq k}(\D_{\bullet}(X,Y))\supseteq\FF^{\geq k'}(\D_{\bullet}(X,Y))\supseteq\cdots\supseteq\D_{\bullet}(X,Y)
    \] for \(k<k'\), in which the number of dimensions for which Dowker duality holds increases by \(1\) at each step.
\end{remark}

Using the fact that the Dowker-Rips complex is the 2-flagification of the Dowker complex, we get the following \emph{Dowker-Rips duality}.

\begin{restatable}{theorem}{ThmMainPractical}\label{thm:main_practical}
    Let \((Z,d)\) be a metric space and let \(X,Y\subseteq Z\) be non-empty and finite disjoint subsets.
    For \(\varepsilon\geq 0\), define the relation \(R_{\varepsilon}\subseteq X\times Y\) by \[
        (x,y)\in R_{\varepsilon}\quad\text{iff}\quad d(x,y)\leq\varepsilon.
    \]
    Denote by \(\DR_{\bullet}(X,Y)\) the filtration given by \(\left\{\DR_{R_{\varepsilon}}(X,Y)\right\}_{\varepsilon\in\R^{+}}\), and similarly for \(\DR_{\bullet}(Y,X)\).
    Then we have that \[
        \PH_{i}(\DR_{\bullet}(X,Y))\cong\PH_{i}(\DR_{\bullet}(Y,X))
    \] for \(i=0,1\).
\end{restatable}

The above result is sharp in the sense that its conclusion does not hold for homological dimensions higher than \(1\) (see Proposition~\ref{prop:counterexample_duality} for such an example).
Nevertheless, the Dowker-Rips duality is a desirable property of the Dowker-Rips complex, since, in practice, persistent homology is often computed only up to homological dimension \(1\) for reasons of computational complexity.
In these homological dimensions, the Dowker-Rips duality thus accelerates the computation of the persistent homology of the Dowker-Rips complex: like in the case of the Dowker complex, this duality allows one to potentially swap the roles of \(X\) and \(Y\) in order to compute the less expensive variant of the two Dowker-Rips complexes.
Indeed, we provide an applied example of a machine learning classification pipeline in which simply replacing the Dowker complex with the Dowker-Rips complex leads to speed gains, all while not negatively affecting classification performance

This paper is organized as follows: in Section~\ref{sec:interleaving}, we construct the multiplicative interleavings proving Theorems~\ref{thm:interleaving_with_dowker} and~\ref{thm:interleaving_with_flipped}.
In Section~\ref{sec:duality}, which is the main technical section, we prove the Dowker-Rips duality (Theorem~\ref{thm:main_practical}).
Finally, in Section~\ref{sec:dowker_rips_as_drop_in_replacement}, we present the application that justifies using the Dowker-Rips complex instead of the Dowker complex in practice.


\section{Multiplicative interleavings}\label{sec:interleaving}

Interleavings are a way to capture similarities of filtrations.
While in many cases \emph{additive interleavings} are desirable, in some cases multiplicative interleavings are the best that can be done.
Following we recall the definition of a multiplicative interleaving (see e.g.~\ \cite{dey_wang_computational_topology,oudot_persistence_theory}).

\begin{definition}\label{def:mult_interleaving}
Let $\mathcal{F}=\{F_a\}_{a\in\R}$ and $\mathcal{G}=\{G_a\}_{a\in\R}$ be filtrations.
We say that $\mathcal{F}$ and $\mathcal{G}$ are \emph{multiplicatively $c$-interleaved} if there are maps $\varphi_a: F_a\rightarrow G_{ca}$ and $\psi_a: G_a\rightarrow F_{ca}$ such that the following diagrams commute for every $a\in\R$ and $\varepsilon>0$:

\[
    \begin{tikzcd}
        U_a\arrow[rr, hookrightarrow]\arrow[rd, "\varphi_a"] & & U_{a+\varepsilon}\arrow[rd, "\varphi_{a+\varepsilon}"] \\
        & V_{ca}\arrow[rr, hookrightarrow] & & V_{c(a+\varepsilon)}
        \end{tikzcd} \text{ and }
        \begin{tikzcd}
        & U_{ca}\arrow[rr, hookrightarrow] & & U_{c(a+\varepsilon)} \\
        V_{a}\arrow[ru, "\psi_a"]\arrow[rr, hookrightarrow] & & V_{a+\varepsilon}\arrow[ru, "\psi_{a+\varepsilon}"]
    \end{tikzcd}
\]

\[
    \begin{tikzcd}
        U_a\arrow[rr, hookrightarrow]\arrow[rd, "\varphi_a"] & & U_{c^2a} \\
        & V_{ca}\arrow[ru, "\psi_{ca}"] &
        \end{tikzcd} \text{ and }
        \begin{tikzcd}
        & U_{ca}\arrow[rd, "\varphi_{ca}"] & \\
        V_a\arrow[ru, "\psi_a"]\arrow[rr, hookrightarrow] & & V_{c^2a}
    \end{tikzcd}
\]
\end{definition}

One prominent example of this is the multiplicative \(2\)-interleaving of the \cechtext filtration and the Vietoris-Rips filtration stemming from the inclusions \[
    \cech_{\varepsilon}(X)\subseteq\vr_{\varepsilon}(X)\subseteq\cech_{2\varepsilon}(X)
\] for \(\varepsilon\geq 0\), as was alluded to in the introduction.

We now prove the existence of the two multiplicative interleavings claimed in Section~\ref{sec:introduction}.

\DowkerInterleaving*

\begin{proof}
    It suffices to show that \[
        \D_{\varepsilon}(X,Y)\subseteq\DR_{\varepsilon}(X,Y)\subseteq\D_{3\varepsilon}(X,Y)
    \] for all \(\varepsilon\geq 0\); by defining $\varphi_{\varepsilon}$ and $\psi_{\varepsilon}$ as inclusions, the commutativity of the required diagrams then follows immediately.

    Let \(\varepsilon\geq 0\).
    The inclusion \(\D_{\varepsilon}(X,Y)\subseteq\DR_{\varepsilon}(X,Y)\) is immediate from the definition of \(\DR_{\varepsilon}(X,Y)\) as the flagification of \(\D_{\varepsilon}(X,Y)\).

    Suppose now that \(\DR_{\varepsilon}(X,Y)\) contains some simplex \(\sigma=[x_{0},\dots,x_{n}]\), where \(x_{0},\dots,x_{n}\in X\).
    By definition, this means that for any \(x_{i},x_{j}\in\sigma\) there exists an element \(y_{ij}\in Y\) such that \(d(x_{i},y_{ij})\leq\varepsilon\) and \(d(x_{j},y_{ij})\leq\varepsilon\).
    Now, given any \(x_{i}\in\sigma\), we have that
    \begin{align*}
        d(x_{i},y_{kl})&\leq d(x_{i},x_{k})+d(x_{k},y_{kl})\\
        &\leq d(x_{i},y_{ki})+d(y_{ki},x_{k})+d(x_{k},y_{kl})\\
        &\leq 3\varepsilon
    \end{align*}
    for any \(0\leq k<j\leq n\).
    Hence \(\sigma\in\D_{3\varepsilon}(X,Y)\), as claimed.
\end{proof}

\FlipInterleaving*

\begin{proof}
    Consider the following chain of maps
    \begin{equation*}
        \begin{tikzcd}
            \DR_{\varepsilon}(X,Y) \arrow[r, hook, "\iota_{\DR,\D}^{\varepsilon}"] &
            \D_{3\varepsilon}(X,Y) \arrow[r, "\iota^{(1)}"] &
            \D_{3\varepsilon}^{(1)}(X,Y) \arrow[r, "\Gamma"] &
            \D_{3\varepsilon}(Y,X) \arrow[r, hook, "\iota_{\D,\DR}^{3\varepsilon}"] &
            \DR_{3\varepsilon}(Y,X),
        \end{tikzcd}
    \end{equation*}
    where $\iota_{\DR,\D}^{\varepsilon}$ and $\iota_{\D,\DR}^{3\varepsilon}$ denote the inclusion maps from Theorem~\ref{thm:interleaving_with_dowker}, $\iota^{(1)}$ denotes the inclusion of the respective complex into its first barycentric subdivision, and where \(\Gamma\) denotes the simplicial map from~\cite{chowdhury_memoli_functorial_dowker_theorem}.
    We define $\varphi_{\varepsilon}:=\iota_{\D,\DR}^{3\varepsilon}\circ\Gamma\circ\iota^{(1)}\circ\iota_{\DR,\D}^{\varepsilon}$.
    The functions $\psi_{\varepsilon}$ are defined symmetrically.

    Consider first the following diagram:
    \begin{equation}\label{eq:interleaving_with_flipped_triangle}
        \begin{tikzcd}
            \DR_{\varepsilon}(X,Y) \arrow[rrrr, hook] \arrow[dr, hook, "\iota_{\DR,\D}^{\varepsilon}"] & & & & \DR_{9\varepsilon}(X,Y) \\
            & \D_{3\varepsilon}(X,Y) \arrow[d, hook, "\iota^{(1)}"] & & \D_{9\varepsilon}(X,Y) \arrow[ur, hook, "\iota_{\D,\DR}^{9\varepsilon}"] & \\
            & \D_{3\varepsilon}^{(1)}(X,Y) \arrow[d, "\Gamma"] & & \D_{9\varepsilon}^{(1)}(X,Y) \arrow[u, "\Gamma"] & \\
            & \D_{3\varepsilon}(Y,X) \arrow[dr, hook, "\iota_{\D,\DR}^{3\varepsilon}"] & & \D_{9\varepsilon}(Y,X) \arrow[u, hook, "\iota^{(1)}"] & \\
            & & \DR_{3\varepsilon}(Y,X) \arrow[ur, hook, "\iota_{\DR,\D}^{3\varepsilon}"] & &
        \end{tikzcd}
    \end{equation}
    By definition of $\varphi_{\varepsilon}$ and $\psi_{\varepsilon}$, this is exactly the triangular diagram required for multiplicative interleavings.
    It follows from functoriality of $\Gamma$ established in \cite{chowdhury_memoli_functorial_dowker_theorem} together with the fact that all other maps are inclusion maps that this diagram commutes.

    Similarly, the relevant trapezoidal diagram is the following:
    \begin{equation}\label{eq:interleaving_with_flipped_trapezoid}
        \begin{tikzcd}
            \DR_{\varepsilon}(X,Y) \arrow[rrr, hook] \arrow[dr, hook, "\iota_{\DR,\D}^{\varepsilon}"] & & & \DR_{\varepsilon+\varepsilon'}(X,Y)\arrow[dr, hook, "\iota_{\DR,\D}^{\varepsilon+\varepsilon'}"] & \\
            & \D_{3\varepsilon}(X,Y) \arrow[d, hook, "\iota^{(1)}"] & & & \D_{3(\varepsilon+\varepsilon')}(X,Y) \arrow[d, hook, "\iota^{(1)}"] \\
            & \D_{3\varepsilon}^{(1)}(X,Y) \arrow[d, "\Gamma"] & & & \D_{3(\varepsilon+\varepsilon')}^{(1)}(X,Y) \arrow[d, "\Gamma"] \\
            & \D_{3\varepsilon}(Y,X) \arrow[d, hook, "\iota_{\D,\DR}^{3\varepsilon}"] & & & \D_{3(\varepsilon+\varepsilon')}(Y,X) \arrow[d, hook, "\iota_{\D,\DR}^{3(\varepsilon+\varepsilon')}"] \\
            & \DR_{3\varepsilon}(X,Y) \arrow[rrr, hook] & & & \DR_{3(\varepsilon+\varepsilon')}(X,Y)
        \end{tikzcd}
    \end{equation}
    Again, this diagram commutes by functoriality of $\Gamma$ and the fact that all other maps are inclusion maps.
\end{proof}

We conclude this section by providing an example illustrating that the interleaving from Theorem~\ref{thm:interleaving_with_dowker} is sharp in the sense that the inclusion \(\DR_{\varepsilon}(X,Y)\subseteq\D_{3\varepsilon}(X,Y)\) does not hold when \(3\) is replaced by some value \(c<3\).

\begin{proposition}\label{prop:counterexample_interleaving}
    There exists a setting for Theorem~\ref{thm:interleaving_with_dowker} such that \[
        \DR_{\varepsilon}(X,Y)\not\subseteq\D_{c\varepsilon}(X,Y)
    \] for any \(c<3\).
\end{proposition}

\begin{proof}
    Define \((Z,d)\) as the graph pictured in Figure~\ref{fig:counterexample_interleaving} equipped with the shortest-path metric, and let \(X=\left\{x_{0},x_{1},x_{2}\right\}\subseteq Z\) and \(Y=\left\{y_{0},y_{1},y_{2}\right\}\subseteq Z\) be the set of the crossed and hollow circles, respectively.
    It is easy to see that \([x_{i},x_{j}]\in\D_{1}(X,Y)\) for all \(0\leq i<j\leq 2\), and hence that \([x_{0},x_{1},x_{2}]\in\DR_{1}(X,Y)\).
    In contrast, for \(\D_{c}(X,Y)\), \(c\geq 1\), to contain \([x_{0},x_{1},x_{2}]\), \(c\) must be large enough to guarantee the existence of an element \(y\in Y\) such that \(d(y,x_{i})\leq c\) for all \(0\leq i\leq 2\).
    Since \(d(y_{i},x_{i})=3\) for all \(0\leq i\leq 2\), this is the case only if \(c\geq 3\).
\end{proof}

\begin{figure}[ht]
    \centering
    \def\svgwidth{0.45\textwidth}
\begingroup%
  \makeatletter%
  \providecommand\color[2][]{%
    \errmessage{(Inkscape) Color is used for the text in Inkscape, but the package 'color.sty' is not loaded}%
    \renewcommand\color[2][]{}%
  }%
  \providecommand\transparent[1]{%
    \errmessage{(Inkscape) Transparency is used (non-zero) for the text in Inkscape, but the package 'transparent.sty' is not loaded}%
    \renewcommand\transparent[1]{}%
  }%
  \providecommand\rotatebox[2]{#2}%
  \newcommand*\fsize{\dimexpr\f@size pt\relax}%
  \newcommand*\lineheight[1]{\fontsize{\fsize}{#1\fsize}\selectfont}%
  \ifx\svgwidth\undefined%
    \setlength{\unitlength}{269.23722875bp}%
    \ifx\svgscale\undefined%
      \relax%
    \else%
      \setlength{\unitlength}{\unitlength * \real{\svgscale}}%
    \fi%
  \else%
    \setlength{\unitlength}{\svgwidth}%
  \fi%
  \global\let\svgwidth\undefined%
  \global\let\svgscale\undefined%
  \makeatother%
  \begin{picture}(1,0.72502441)%
    \lineheight{1}%
    \setlength\tabcolsep{0pt}%
    \put(0,0){\includegraphics[width=\unitlength,page=1]{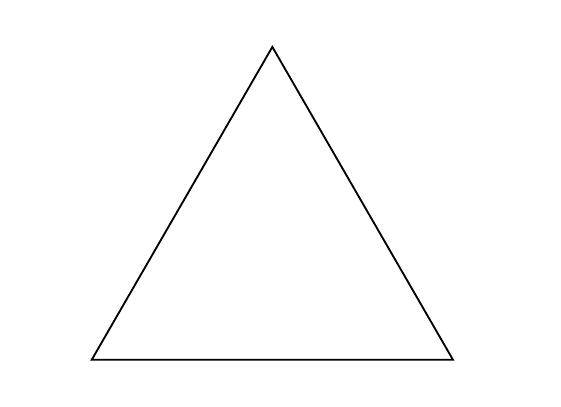}}%
    \put(0.47125807,0.67846971){\color[rgb]{0,0,0}\makebox(0,0)[lt]{\lineheight{0}\smash{\begin{tabular}[t]{l}\(x_{0}\)\end{tabular}}}}%
    \put(0.23701445,0.35371437){\color[rgb]{0,0,0}\makebox(0,0)[lt]{\lineheight{0}\smash{\begin{tabular}[t]{l}\(y_{2}\)\end{tabular}}}}%
    \put(0.68163151,0.35371437){\color[rgb]{0,0,0}\makebox(0,0)[lt]{\lineheight{0}\smash{\begin{tabular}[t]{l}\(y_{1}\)\end{tabular}}}}%
    \put(0.14925686,0.01501475){\color[rgb]{0,0,0}\makebox(0,0)[lt]{\lineheight{0}\smash{\begin{tabular}[t]{l}\(x_{1}\)\end{tabular}}}}%
    \put(0.47125117,0.01500519){\color[rgb]{0,0,0}\makebox(0,0)[lt]{\lineheight{0}\smash{\begin{tabular}[t]{l}\(y_{0}\)\end{tabular}}}}%
    \put(0.79325993,0.01501475){\color[rgb]{0,0,0}\makebox(0,0)[lt]{\lineheight{0}\smash{\begin{tabular}[t]{l}\(x_{2}\)\end{tabular}}}}%
    \put(0,0){\includegraphics[width=\unitlength,page=2]{interleaving_counterexample.pdf}}%
  \end{picture}%
\endgroup%

    \caption{The metric space \((Z,d)\) from the proof of Proposition~\ref{prop:counterexample_interleaving}, with subsets \(X\) and \(Y\) consisting of the crossed and hollow circles, respectively.}
    \label{fig:counterexample_interleaving}
\end{figure}


\section{Dowker-Rips duality}\label{sec:duality}

In this section, we prove the strengthenings of the interleaving results from Section~\ref{sec:interleaving} and, in particular, the Dowker-Rips duality.
To do so, we restate and extend the definition of \(k\)-flagification to include a partial flagification that is needed in the proofs.

\begin{definition}\label{def:k_flagification_extended}
    Given a simplicial complex \(X\), the \emph{flagification of \(X\)}, denoted by \(\FF(X)\), is defined as the simplicial complex that is obtained from \(X\) by including a simplex \(\sigma\subseteq X\) whenever all edges of \(\sigma\) already belong to \(X\) and \(\dim(\sigma)\geq 2\).
    More generally, for an integer \(k\geq 2\), the \emph{\(k\)-flagification of \(X\)}, denoted by \(\FF^{\geq k}(X)\), is defined as the complex that is obtained from \(X\) by including a simplex \(\sigma\subseteq X\) whenever all \((k-1)\)-dimensional faces of \(\sigma\) already belong to \(X\) and \(\dim(\sigma)\geq k\).
    Finally, the \emph{partial \(k\)-flagification of \(X\)}, denoted by \(\FF^{k}(X)\), is defined as the complex that is obtained from \(X\) by including a simplex \(\sigma\subseteq X\) whenever all \((k-1)\)-dimensional faces of \(\sigma\) already belong to \(X\) and \(\dim(\sigma)=k\).
\end{definition}

Recall from~\cite[Section 5.1]{chowdhury_memoli_functorial_dowker_theorem} that there exists a simplicial map \(\Gamma\colon\D_{R}^{(1)}(X,Y)\to\D_{R}(Y,X)\) that induces a homotopy equivalence \(\psi\colon\vert\D_{R}^{(1)}(X,Y)\vert\to\vert\D_{R}(Y,X)\vert\) on the level of geometric realizations.
Here and in what follows, \(X^{(1)}\) denotes the first barycentric subdivision of a simplicial complex \(X\).
The map \(\Gamma\) is defined by mapping any vertex \(\sigma=[x_{0},\dots,x_{n}]\in\D_{R}^{(1)}(X,Y)\), \(x_{0},\dots,x_{n}\in X\), to an element \(y_{\sigma}\in Y\) such that \((x_{k},y_{\sigma})\in R\) for all \(k=0,\dots,n\).
It is shown in~\cite{chowdhury_memoli_functorial_dowker_theorem} that the map \(\Gamma\) thus defined is simplicial and, moreover, that different choices of \(y_{\sigma}\) in its definition result in maps that are contiguous to one another (and hence induce homotopic maps on the level of geometric realizations).
At a high level, we prove Theorem~\ref{thm:main_practical} by first showing in Lemma~\ref{lemma:main} that the map \(\psi\) can be extended to a map between the partial \(k\)-flagifications.
From this we deduce Proposition~\ref{prop:main}, the main technical result that establishes properties of the extensions of \(\psi\) pertaining to homology and commutativity.
Finally, Theorems~\ref{thm:main_technical} and~\ref{thm:main_practical} will be relatively straight forward consequences of that proposition.

To make sense of the setup of Lemma~\ref{lemma:main}, observe that \(\D_{R}(X,Y)\) is a subcomplex of \(\FF^{k}(\D_{R}(X,Y))\), which implies that \(\D_{R}^{(1)}(X,Y)\) is a subcomplex of \(\FF^{k}(\D_{R}(X,Y))^{(1)}\) for \(k\geq 2\).

\begin{lemma}\label{lemma:main}
    The homotopy equivalence \(\psi\colon\vert\D_{R}^{(1)}(X,Y)\vert\to\vert\D_{R}(Y,X)\vert\) extends to a continuous map \[
        \varphi\colon\vert\FF^{k}(\D_{R}(X,Y))^{(1)}\vert\to\vert\FF^{k}(\D_{R}(Y,X))\vert
    \] for any \(k\geq 2\).
\end{lemma}

\begin{proof}
    To prove the lemma, we must define \(\varphi\) on the portion of \(\vert\FF^{k}(\D_{R}(X,Y))^{(1)}\vert\) that is not present in \(\vert\D_{R}^{(1)}(X,Y)\vert\).
    This portion consists of the geometric realizations of those simplices that belong to \(\FF^{k}(\D_{R}(X,Y))\), but not to \(\D_{R}(X,Y)\).
    Let \(\sigma\in\FF^{k}(\D_{R}(X,Y))\setminus\D_{R}(X,Y)\) be such a simplex.
    Since \(\sigma\) is \(k\)-dimensional, we may write \(\sigma=[x_{0},\dots,x_{k}]\) for some \(x_{0},\dots,x_{k}\in X\).
    Moreover, by definition of \(\FF^{k}(\D_{R}(X,Y))\), it must be the case that all proper faces of \(\sigma\) belong to \(\D_{R}(X,Y)\).
    Letting \(\II_{k}\) denote the set of subsets \(I\subseteq\left\{0,\dots,k\right\}\) such that \(0<\vert I\vert<k+1\), we thus have that \([x_{i}]_{i\in I}\in\D_{R}(X,Y)\) for all \(I\in\II_{k}\).
    Given \(I\in\II_{k}\), let \(x_{I}\in\D_{R}^{(1)}(X,Y)\) denote the vertex corresponding to the face \([x_{i}]_{i\in I}\) of \(\sigma\), and define the subcomplex \(C_{\partial\sigma}^{X}\subseteq\D_{R}^{(1)}(X,Y)\) as the barycentric subdivision of the complex consisting of the proper faces of \(\sigma\).
    Similarly, define \(C_{\sigma}^{X}\subseteq\FF^{k}(\D_{R}(X,Y))^{(1)}\) as the barycentric subdivision of \(\sigma\).
    See Figure~\ref{fig:lemma_main_1} for a schematic illustration of \(C_{\partial\sigma}^{X}\) and \(C_{\sigma}^{X}\) in the case where \(k=2\).

    Given any \(I\in\II_{k}\), set \(y_{I}\coloneqq\Gamma(x_{I})\in\D_{R}(Y,X)\).
    Note that a collection of these elements spans a simplex \([y_{I_{1}},\dots,y_{I_{l}}]\in\D_{R}(Y,X)\) whenever \(I_{1},\dots,I_{l}\in\II_{k}\) are such that \(I_{1}\cap\cdots\cap I_{l}\neq\varnothing\).
    To see this, let \(I_{1},\dots,I_{l}\in\II_{k}\) be such sets.
    Then, by definition of \(\Gamma\), we have that \((x_{i},y_{I_{1}}),\dots,(x_{i},y_{I_{l}})\in R\) for all \(i\in I_{1}\cap\cdots\cap I_{l}\), and hence that \([y_{I_{1}},\dots,y_{I_{l}}]\in\D_{R}(Y,X)\).\footnote{
        Note that the elements \(y_{I}\in Y\) for \(I\in\II_{k}\) are not necessarily pairwise distinct: if \(I,J\in\II_{k}\) are such that \(I\subseteq J\), it can be the case that \(y_{I}=y_{J}\in Y\), in which case the edge \([y_{I},y_{J}]\) degenerates to a point.
    }
    In particular, we have that \(\D_{R}(Y,X)\) contains the \(k+1\) simplices \([y_{I}]_{\left\{I\in\II_{k}\mid i\in I,\vert I\vert=k\right\}}\), each of dimension \(k-1\), for all \(i=0,\dots,k\).
    Hence \(\FF^{k}(\D_{R}(Y,X))\) contains the \(k\)-dimensional simplex \([y_{I}]_{\left\{I\in\II_{k}\mid\vert I\vert=k\right\}}\).
    With this at hand, define the subcomplex \(C_{\partial\sigma}^{Y}\subseteq\D_{R}(Y,X)\) as having vertex set \(\left\{y_{I}\mid I\in\II_{k}\right\}\) and simplices \([y_{I_{1}},\dots,y_{I_{l}}]\), for \(I_{1},\dots,I_{l}\in\II_{k}\) such that \(I_{1}\cap\cdots\cap I_{l}\neq\varnothing\).
    Furthermore, define \(C_{\sigma}^{Y}\subseteq\FF^{k}(\D_{R}(Y,X))\) to be the complex obtained from \(C_{\partial\sigma}^{Y}\) by adding the simplex \([y_{I}]_{\left\{I\in\II_{k}\mid\vert I\vert=k\right\}}\).
    See Figure~\ref{fig:lemma_main_2} for a schematic illustration of \(C_{\partial\sigma}^{Y}\) and \(C_{\sigma}^{Y}\) in the case where \(k=2\).

    By construction, we have that \(\Gamma(C_{\partial\sigma}^{X})\subseteq C_{\partial\sigma}^{Y}\), and hence, by passing to geometric realizations, that \(\psi(\vert C_{\partial\sigma}^{X}\vert)\subseteq\vert C_{\partial\sigma}^{Y}\vert\subseteq\vert C_{\sigma}^{Y}\vert\).
    It remains to show that \(\psi\) extends from \(\vert C_{\partial\sigma}^{X}\vert\) to \(\vert C_{\sigma}^{X}\vert\), for which, in turn, it suffices to show that \(\vert C_{\sigma}^{Y}\vert\) is contractible (see, for instance,~\cite[Corollary 4.73]{hatcher_algebraic_topology}).
    To that end, observe that for any \(i=0,\dots,k\), the simplex \([y_{I}]_{\left\{I\in\II_{k}\mid i\in I\right\}}\in C_{\sigma}^{Y}\), that is, the simplex induced by all \(y_{I}\) whose subscript contains \(i\), is a maximal face of \(C_{\sigma}^{Y}\).
    Indeed, \([y_{I}]_{\left\{I\in\II_{k}\mid i\in I\right\}}\) is the only maximal face containing the vertex \(y_{i}\), and hence the latter vertex is a free face of \(C_{\sigma}^{Y}\).
    We may thus collapse \(C_{\sigma}^{Y}\) with respect to the free faces \(y_{0},\dots,y_{k}\), which results in a complex homotopy equivalent to \(C_{\sigma}^{Y}\).
    This resulting complex is the subcomplex of \(C_{\sigma}^{Y}\) induced by the vertices \(y_{I}\) for \(I\in\II_{k}\) and \(\vert I\vert>1\).
    Similarly to before, all vertices of this new complex that are of the form \(y_{I}\) for \(I\in\II_{k}\) and \(\vert I\vert=2\) are free faces.
    We may thus collapse this complex with respect to these free faces to obtain a complex that is still homotopy equivalent to \(C_{\sigma}^{Y}\).
    Repeating this process eventually results in the subcomplex of \(C_{\sigma}^{Y}\) induced by the vertices \(y_{I}\) for \(I\in\II_{k}\) and \(\vert I\vert=k\), and demonstrates that this resulting complex is homotopy equivalent to the original complex \(C_{\sigma}^{Y}\).
    As we have seen in the previous paragraph, we have that \([y_{I}]_{\left\{I\in\II_{k}\mid\vert I\vert=k\right\}}\in\FF^{k}(\D_{R}(Y,X))\).
    In other words, the complex resulting from iteratively collapsing as above is simply a \(k\)-dimensional simplex and hence \(C_{\sigma}^{Y}\), being homotopy equivalent to a simplex, is contractible.
\end{proof}

\begin{figure}[ht]
    \centering
    \begin{subfigure}[t]{0.45\textwidth}
        \centering
        \def\svgwidth{\linewidth}
\begingroup%
  \makeatletter%
  \providecommand\color[2][]{%
    \errmessage{(Inkscape) Color is used for the text in Inkscape, but the package 'color.sty' is not loaded}%
    \renewcommand\color[2][]{}%
  }%
  \providecommand\transparent[1]{%
    \errmessage{(Inkscape) Transparency is used (non-zero) for the text in Inkscape, but the package 'transparent.sty' is not loaded}%
    \renewcommand\transparent[1]{}%
  }%
  \providecommand\rotatebox[2]{#2}%
  \newcommand*\fsize{\dimexpr\f@size pt\relax}%
  \newcommand*\lineheight[1]{\fontsize{\fsize}{#1\fsize}\selectfont}%
  \ifx\svgwidth\undefined%
    \setlength{\unitlength}{267.25428303bp}%
    \ifx\svgscale\undefined%
      \relax%
    \else%
      \setlength{\unitlength}{\unitlength * \real{\svgscale}}%
    \fi%
  \else%
    \setlength{\unitlength}{\svgwidth}%
  \fi%
  \global\let\svgwidth\undefined%
  \global\let\svgscale\undefined%
  \makeatother%
  \begin{picture}(1,0.75822414)%
    \lineheight{1}%
    \setlength\tabcolsep{0pt}%
    \put(0,0){\includegraphics[width=\unitlength,page=1]{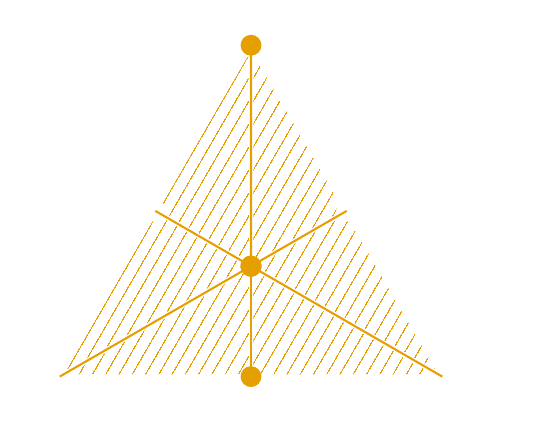}}%
    \put(0.41752174,0.7163449){\color[rgb]{0,0,0}\makebox(0,0)[lt]{\lineheight{0}\smash{\begin{tabular}[t]{l}\(x_{0}\)\end{tabular}}}}%
    \put(0.79786467,0.01476603){\color[rgb]{0,0,0}\makebox(0,0)[lt]{\lineheight{0}\smash{\begin{tabular}[t]{l}\(x_{2}\)\end{tabular}}}}%
    \put(0.05986801,0.01476603){\color[rgb]{0,0,0}\makebox(0,0)[lt]{\lineheight{0}\smash{\begin{tabular}[t]{l}\(x_{1}\)\end{tabular}}}}%
    \put(0.44594269,0.01476603){\color[rgb]{0,0,0}\makebox(0,0)[lt]{\lineheight{0}\smash{\begin{tabular}[t]{l}\(x_{12}\)\end{tabular}}}}%
    \put(0.16142285,0.40765066){\color[rgb]{0,0,0}\makebox(0,0)[lt]{\lineheight{0}\smash{\begin{tabular}[t]{l}\(x_{01}\)\end{tabular}}}}%
    \put(0.65814294,0.40765066){\color[rgb]{0,0,0}\makebox(0,0)[lt]{\lineheight{0}\smash{\begin{tabular}[t]{l}\(x_{02}\)\end{tabular}}}}%
    \put(0,0){\includegraphics[width=\unitlength,page=2]{lemma_main_1.pdf}}%
  \end{picture}%
\endgroup%

        \caption{The complexes \(C_{\partial\sigma}^{X}\) (black) and \(C_{\sigma}^{X}\) (black and orange).}
        \label{fig:lemma_main_1}
    \end{subfigure}
    \hfill
    \begin{subfigure}[t]{0.45\textwidth}
        \centering
        \def\svgwidth{\linewidth}
\begingroup%
  \makeatletter%
  \providecommand\color[2][]{%
    \errmessage{(Inkscape) Color is used for the text in Inkscape, but the package 'color.sty' is not loaded}%
    \renewcommand\color[2][]{}%
  }%
  \providecommand\transparent[1]{%
    \errmessage{(Inkscape) Transparency is used (non-zero) for the text in Inkscape, but the package 'transparent.sty' is not loaded}%
    \renewcommand\transparent[1]{}%
  }%
  \providecommand\rotatebox[2]{#2}%
  \newcommand*\fsize{\dimexpr\f@size pt\relax}%
  \newcommand*\lineheight[1]{\fontsize{\fsize}{#1\fsize}\selectfont}%
  \ifx\svgwidth\undefined%
    \setlength{\unitlength}{283.80775992bp}%
    \ifx\svgscale\undefined%
      \relax%
    \else%
      \setlength{\unitlength}{\unitlength * \real{\svgscale}}%
    \fi%
  \else%
    \setlength{\unitlength}{\svgwidth}%
  \fi%
  \global\let\svgwidth\undefined%
  \global\let\svgscale\undefined%
  \makeatother%
  \begin{picture}(1,0.91077933)%
    \lineheight{1}%
    \setlength\tabcolsep{0pt}%
    \put(0,0){\includegraphics[width=\unitlength,page=1]{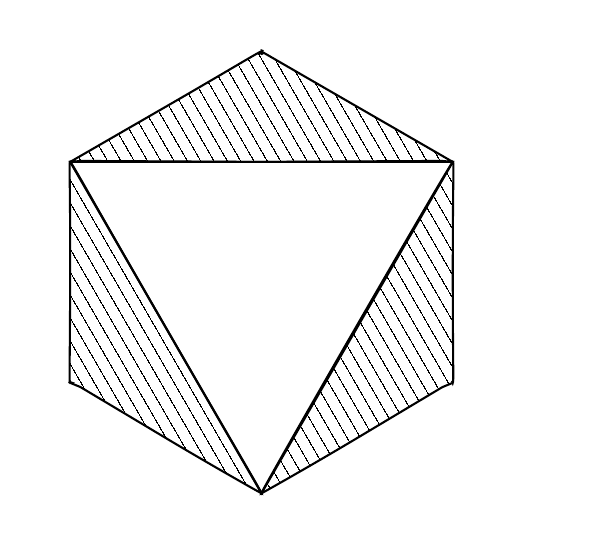}}%
    \put(0.41929793,0.87310451){\color[rgb]{0,0,0}\makebox(0,0)[lt]{\lineheight{0}\smash{\begin{tabular}[t]{l}\(y_{0}\)\end{tabular}}}}%
    \put(0.41929793,0.0121431){\color[rgb]{0,0,0}\makebox(0,0)[lt]{\lineheight{0}\smash{\begin{tabular}[t]{l}\(y_{12}\)\end{tabular}}}}%
    \put(0.02290281,0.24772961){\color[rgb]{0,0,0}\makebox(0,0)[lt]{\lineheight{0}\smash{\begin{tabular}[t]{l}\(y_{1}\)\end{tabular}}}}%
    \put(0.00176176,0.67055104){\color[rgb]{0,0,0}\makebox(0,0)[lt]{\lineheight{0}\smash{\begin{tabular}[t]{l}\(y_{01}\)\end{tabular}}}}%
    \put(0.79983723,0.24772961){\color[rgb]{0,0,0}\makebox(0,0)[lt]{\lineheight{0}\smash{\begin{tabular}[t]{l}\(y_{2}\)\end{tabular}}}}%
    \put(0.79983723,0.61769836){\color[rgb]{0,0,0}\makebox(0,0)[lt]{\lineheight{0}\smash{\begin{tabular}[t]{l}\(y_{02}\)\end{tabular}}}}%
    \put(0,0){\includegraphics[width=\unitlength,page=2]{lemma_main_2.pdf}}%
  \end{picture}%
\endgroup%

        \caption{The complexes \(C_{\partial\sigma}^{Y}\) (black) and \(C_{\sigma}^{Y}\) (black and orange).}
        \label{fig:lemma_main_2}
    \end{subfigure}
    \caption{Schematics accompanying the proof of Lemma~\ref{lemma:main} for the case where \(k=2\).}
    \label{fig:lemma_main}
\end{figure}

With the previous lemma at hand, we can now deduce the required properties of the extensions of the map \(\psi\).

\begin{proposition}\label{prop:main}
    Let \(X\) and \(Y\) be two finite sets, let \(R\subseteq R'\subseteq X\times Y\) be two non-empty relations, and let \(k\geq 2\) an integer.
    Then there exist continuous maps \(\varphi\colon\vert\FF^{k}(\D_{R}(X,Y))\vert\to\vert\FF^{k}(\D_{R}(Y,X))\vert\) and \(\varphi'\colon\vert\FF^{k}(\D_{R'}(X,Y))\vert\to\vert\FF^{k}(\D_{R'}(Y,X))\vert\) that induce isomorphisms on the level of \(i\)-dimensional homology for \(i=0,\dots,k-1\), and, moreover, such that the diagram
    \begin{equation}\label{eq:commutativity_of_phi}
        \begin{tikzcd}
            \vert\FF^{k}(\D_{R}(X,Y))\vert \arrow[r, hook] \arrow[d, swap, "\varphi"] & \vert\FF^{k}(\D_{R'}(X,Y))\vert \arrow[d, "\varphi'"] \\
            \vert\FF^{k}(\D_{R}(Y,X))\vert \arrow[r, hook] & \vert\FF^{k}(\D_{R'}(Y,X))\vert
        \end{tikzcd}
    \end{equation}
    commutes up to homotopy.
    Here, the horizontal maps are given by inclusion.
\end{proposition}

\begin{proof}
    Let \(\varphi\colon\vert\FF^{k}(\D_{R}(X,Y))^{(1)}\vert\to\vert\FF^{k}(\D_{R}(Y,X))\vert\) be an extension of the homotopy equivalence \(\psi\colon\vert\D_{R}^{(1)}(X,Y)\vert\to\vert\D_{R}(Y,X)\vert\), whose existence is guaranteed by Lemma~\ref{lemma:main}.

    We first show that \(\varphi\) induces isomorphisms on the level of \(i\)-dimensional homology for \(i=0,\dots,k-1\).
    To that end, consider the commutative diagram \[
        \begin{tikzcd}
            \vert\D_{R}(X,Y)\vert \arrow[r, hook, "\iota^{X}"] \arrow[d, "\psi", swap] & \vert\FF^{k}(\D_{R}(X,Y))\vert \arrow[d, "\varphi"] \\
            \vert\D_{R}(Y,X)\vert \arrow[r, hook, "\iota^{Y}"] & \vert\FF^{k}(\D_{R}(Y,X))\vert
        \end{tikzcd}
    \] where \(\iota^{X}\) and \(\iota^{Y}\) denote inclusion maps, and where we identified \(\vert\D_{R}^{(1)}(X,Y)\) and \(\vert\D_{R}(X,Y)\vert\) via the canonical homeomorphism between them.
    Now, since \(\vert\FF^{k}(\D_{R}(Y,X))\vert\) is obtained from \(\vert\D_{R}(Y,X)\vert\) by attaching \(k\)-dimensional cells, it follows that \(\iota^{Y}\) induces an isomorphism on the level of \(i\)-dimensional homology for \(i=0,\dots,k-2\), and a surjection on the level of \((k-1)\)-dimensional homology.
    Hence, using the fact that \(\psi\) is a homotopy equivalence, we have that the map \(\iota^{Y}\circ\psi\) induces a surjection on the level of \(i\)-dimensional homology for \(i=0,\dots,k-1\).
    By commutativity of the above diagram, the same is true about the map \(\varphi\circ\iota^{X}\), and hence the map that \(\varphi\) alone induces on the level of \(i\)-dimensional homology must be a surjection, too, for \(i=0,\dots,k-1\).
    Swapping the roles of \(X\) and \(Y\) in the above, it follows that \(\HH_{i}(\vert\FF^{k}(\D_{R}(X,Y))\vert)\) surjects onto \(\HH_{i}(\vert\FF^{k}(\D_{R}(Y,X))\vert)\) and vice versa for \(i=0,\dots,k-1\).
    Since all simplicial complexes involved are finite, we thus have that \(\HH_{i}(\vert\FF^{k}(\D_{R}(X,Y))\vert)\cong\HH_{i}(\vert\FF^{k}(\D_{R}(Y,X))\vert)\), and hence that \(\varphi\) induces isomorphisms on the level of \(i\)-dimensional homology for \(i=0,\dots,k-1\), as claimed.

    To prove commutativity of Diagram~\ref{eq:commutativity_of_phi} in the statement of Proposition~\ref{prop:main}, consider the following diagram
    \begin{equation}\label{eq:commutativity_of_phi_proof}
        \begin{tikzcd}
            \vert\FF^{k}(\D_{R}(X,Y))\vert \arrow[rrr, hook, "\iota_{R,R'}^{X}"] \arrow[ddd, "\varphi", swap] & & & \vert\FF^{k}(\D_{R'}(X,Y))\vert \arrow[ddd, "\varphi'"] \\
            & \vert\D_{R}(X,Y)\vert \arrow[r, hook] \arrow[d, "\psi", swap] \arrow[lu, hook'] & \vert\D_{R'}(X,Y)\vert \arrow[d, "\psi'"] \arrow[ru, hook] & \\
            & \vert\D_{R}(Y,X)\vert \arrow[r, hook] \arrow[ld, hook'] & \D_{R'}(Y,X)\vert \arrow[rd, hook] & \\
            \vert\FF^{k}(\D_{R}(Y,X))\vert \arrow[rrr, hook, "\iota_{R,R'}^{Y}"] & & & \vert\FF^{k}(\D_{R'}(Y,X))\vert
        \end{tikzcd}
    \end{equation}
    where \(\varphi\) and \(\varphi'\) are extensions of the homotopy equivalences \(\psi\) and \(\psi'\), respectively, as before; where hooked arrows denote inclusion maps; and where we identified \(\vert\D_{R}(X,Y)\vert\) and \(\vert\D_{R}(X,Y)^{(1)}\vert\) as before.
    Observe that the upper and lower trapezoids are commutative because the respective maps are inclusion maps, while commutativity of the left and right trapezoids follows from the fact that \(\varphi\) and \(\varphi'\) are extensions of \(\psi\) and \(\psi'\), respectively.
    Moreover, the inner rectangle commutes up to homotopy by~\cite[Theorem 3]{chowdhury_memoli_functorial_dowker_theorem} and we may thus assume its precise commutativity.\footnote{
        Precise commutativity of this rectangle is achieved by making the choices of \(y_{\sigma}\) in the definition of the maps \(\psi\) and \(\psi'\) in a consistent manner.
    }

    Now, let \(x\in\vert\FF^{k}(\D_{R}(X,Y))\vert\).
    If \(x\in\vert\D_{R}(X,Y)\vert\subseteq\vert\FF^{k}(\D_{R}(X,Y))\vert\), then the fact that \((\iota_{R,R'}^{Y}\circ\varphi)(x)=(\varphi'\circ\iota_{R,R'}^{X})(x)\) is an immediate consequence of commutativity of the trapezoids and the inner rectangle in Diagram~\ref{eq:commutativity_of_phi_proof}.
    Suppose now that \(x\in\vert\FF^{k}(\D_{R}(X,Y))\vert\setminus\vert\D_{R}(X,Y)\vert\), so that \(x\) belongs to the geometric realization of some simplex \(\sigma_{x}\) that is present in \(\FF^{k}(\D_{R}(X,Y))\) but not in \(\D_{R}(X,Y)\).
    Note that the extensions \(\varphi\) and \(\varphi'\) are constructed from \(\psi\) and \(\psi'\), respectively, on a per simplex basis.
    We may thus assume that \(\varphi'\) agrees with \(\varphi\) on the geometric realizations of simplices stemming that are already present in \(\FF^{k}(\D_{R}(X,Y))\), which establishes the equality \((\iota_{R,R'}^{Y}\circ\varphi)(x)=(\varphi'\circ\iota_{R,R'}^{X})(x)\) in this case.
\end{proof}

We now prove the main theorems, which we restate for convenience.

\ThmMainTechnical*

\begin{proof}[Proof of Theorem~\ref{thm:main_technical}]
    Let \(j,j'\in J\) be such that \(j<j'\), and consider the following diagram of maps
    \begin{equation}\label{eq:proof_cor_main}
        \begin{tikzcd}
            \vert\FF^{\geq k}(\D_{R_{j}}(X,Y))\vert \arrow[r, hook] & \vert\FF^{\geq k}(\D_{R_{j'}}(X,Y))\vert \\
            \vert\FF^{k}(\D_{R_{j}}(X,Y))\vert \arrow[r, hook] \arrow[d, swap, "\varphi"] \arrow[u, hook] & \vert\FF^{k}(\D_{R_{j'}}(X,Y))\vert \arrow[d, "\varphi'"] \arrow[u, hook] \\
            \vert\FF^{k}(\D_{R_{j}}(Y,X))\vert \arrow[r, hook] \arrow[d, hook] & \vert\FF^{k}(\D_{R_{j'}}(Y,X))\vert \arrow[d, hook] \\
            \vert\FF^{\geq k}(\D_{R_{j}}(Y,X))\vert \arrow[r, hook] & \vert\FF^{\geq k}(\D_{R_{j'}}(Y,X))\vert
        \end{tikzcd}
    \end{equation}
    where \(\varphi\) and \(\varphi'\) are maps as in the statement of Proposition~\ref{prop:main} and where hooked arrows denote inclusion maps.
    The top and bottom rectangles are commutative since the maps involved are inclusion maps, and commutativity of the middle rectangle follows Proposition~\ref{prop:main}.

    Since, for instance, \(\FF^{\geq k}(\D_{R_{j}}(X,Y))\) and \(\FF^{k}(\D_{R_{j}}(X,Y))\) share the same \(k\)-skeleton, it follows that the top left inclusion map induces an isomorphism on the level of \(i\)-dimensional homology for \(i=0,\dots,k-1\).
    Similarly, it follows that the same is true for the other vertical inclusion maps, and hence, by Proposition~\ref{prop:main}, for all vertical maps.
    Applying the homology functor to Diagram~\ref{eq:proof_cor_main}, and suppressing the two middle rows, we obtain the commutative diagram
    \begin{equation}\label{eq:proof_cor_main_homology}
        \begin{tikzcd}
            \HH_{i}(\vert\FF^{\geq k}(\D_{R_{j}}(X,Y))\vert) \arrow[r,] \arrow[d, leftrightarrow, "\cong"{rotate=90, anchor=south}] & \HH_{i}(\vert\FF^{\geq k}(\D_{R_{j'}}(X,Y))\vert) \arrow[d, leftrightarrow, "\cong"{rotate=90, anchor=south}] \\
            \HH_{i}(\vert\FF^{\geq k}(\D_{R_{j}}(Y,X))\vert) \arrow[r] & \HH_{i}(\vert\FF^{\geq k}(\D_{R_{j'}}(Y,X))\vert)
        \end{tikzcd}
    \end{equation}
    for \(i=0,\dots,k-1\).
    Diagram~\ref{eq:proof_cor_main_homology} thus establishes an isomorphism of persistence modules \(\PH_{i}(\FF^{\geq k}(\D_{\bullet}(X,Y)))\cong\PH_{i}(\FF^{\geq k}(\D_{\bullet}(Y,X)))\) for \(i=0,\dots,k-1\), as claimed.
\end{proof}

\ThmMainPractical*

\begin{proof}
    This is an immediate consequence of setting \(k=2\) in Theorem~\ref{thm:main_technical}.
\end{proof}

We conclude this section by providing an example illustrating that the Dowker-Rips duality is sharp in the sense that its conclusion does not hold for homological dimensions higher than \(1\).

\begin{proposition}\label{prop:counterexample_duality}
    There exists a setting for Theorem~\ref{thm:main_practical} in which the conclusion fails for \(i=2\).
\end{proposition}

\begin{proof}
    Let \(X=\left\{x_{0},\dots,x_{3}\right\}\subseteq\R^{3}\) denote the set of vertices of a regular tetrahedron with edge length \(1\) embedded in \(\R^{3}\), and let \(Y=\left\{y_{ij}\mid 0\leq i<j\leq 3\right\}\), where \(y_{ij}\) is defined to be the midpoint of \(x_{i}\) and \(x_{j}\), \(0\leq i<j\leq 3\).
    Denote by \(\D_{\bullet}(X,Y)\) the filtration given by \(\left\{\D_{R_{\varepsilon}}(X,Y)\right\}_{\varepsilon\in\R^{+}}\), and similarly for \(\D_{\bullet}(Y,X)\).
    Then we have that \(\D_{1/2}(X,Y)\) is homeomorphic to the geometric realization of \(K_{4}\), the complete graph on four vertices.
    In contrast, the complex \(\D_{1/2}(Y,X)\) has vertex set \(Y\), and a set of vertices spans a simplex precisely when their subscripts share a common element.
    See Figure~\ref{fig:counterexample_duality} for an illustration of the complexes \(\D_{1/2}(X,Y)\) and \(\D_{1/2}(Y,X)\).

    It follows that the flagifications of \(\D_{1/2}(X,Y)\) and \(\D_{1/2}(Y,X)\) equal a \(3\)-simplex and an octahedron, respectively.
    Hence \(\DR_{1/2}(X,Y)\) and \(\DR_{1/2}(Y,X)\) are homotopy equivalent to a point and a \(2\)-sphere, respectively.
    This implies that \[
        \HH_{2}(\DR_{1/2}(X,Y))\cong 0\quad\text{and}\quad\HH_{2}(\DR_{1/2}(Y,X))\cong\Z,
    \] and, in particular, that \[
        \PH_{2}(\DR_{1/2}(X,Y))\not\cong\PH_{2}(\DR_{1/2}(Y,X)),
    \] as claimed.
\end{proof}

\begin{figure}[ht]
    \centering
    \begin{subfigure}[t]{0.45\textwidth}
        \centering
        \def\svgwidth{\linewidth}
\begingroup%
  \makeatletter%
  \providecommand\color[2][]{%
    \errmessage{(Inkscape) Color is used for the text in Inkscape, but the package 'color.sty' is not loaded}%
    \renewcommand\color[2][]{}%
  }%
  \providecommand\transparent[1]{%
    \errmessage{(Inkscape) Transparency is used (non-zero) for the text in Inkscape, but the package 'transparent.sty' is not loaded}%
    \renewcommand\transparent[1]{}%
  }%
  \providecommand\rotatebox[2]{#2}%
  \newcommand*\fsize{\dimexpr\f@size pt\relax}%
  \newcommand*\lineheight[1]{\fontsize{\fsize}{#1\fsize}\selectfont}%
  \ifx\svgwidth\undefined%
    \setlength{\unitlength}{269.23722875bp}%
    \ifx\svgscale\undefined%
      \relax%
    \else%
      \setlength{\unitlength}{\unitlength * \real{\svgscale}}%
    \fi%
  \else%
    \setlength{\unitlength}{\svgwidth}%
  \fi%
  \global\let\svgwidth\undefined%
  \global\let\svgscale\undefined%
  \makeatother%
  \begin{picture}(1,0.72502441)%
    \lineheight{1}%
    \setlength\tabcolsep{0pt}%
    \put(0,0){\includegraphics[width=\unitlength,page=1]{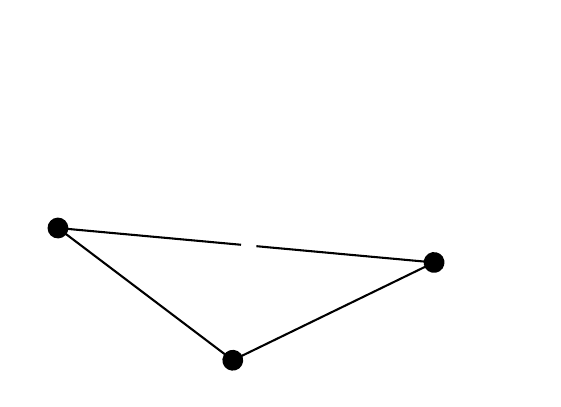}}%
    \put(0.00194629,0.30559514){\color[rgb]{0,0,0}\makebox(0,0)[lt]{\lineheight{0}\smash{\begin{tabular}[t]{l}\(x_{2}\)\end{tabular}}}}%
    \put(0.46906021,0.68340296){\color[rgb]{0,0,0}\makebox(0,0)[lt]{\lineheight{0}\smash{\begin{tabular}[t]{l}\(x_{0}\)\end{tabular}}}}%
    \put(0.80917757,0.23964477){\color[rgb]{0,0,0}\makebox(0,0)[lt]{\lineheight{0}\smash{\begin{tabular}[t]{l}\(x_{3}\)\end{tabular}}}}%
    \put(0.39607367,0.01945957){\color[rgb]{0,0,0}\makebox(0,0)[lt]{\lineheight{0}\smash{\begin{tabular}[t]{l}\(x_{1}\)\end{tabular}}}}%
    \put(0,0){\includegraphics[width=\unitlength,page=2]{duality_counterexample_1.pdf}}%
  \end{picture}%
\endgroup%

    \end{subfigure}
    \hfill
    \begin{subfigure}[t]{0.45\textwidth}
        \centering
        \def\svgwidth{\linewidth}
\begingroup%
  \makeatletter%
  \providecommand\color[2][]{%
    \errmessage{(Inkscape) Color is used for the text in Inkscape, but the package 'color.sty' is not loaded}%
    \renewcommand\color[2][]{}%
  }%
  \providecommand\transparent[1]{%
    \errmessage{(Inkscape) Transparency is used (non-zero) for the text in Inkscape, but the package 'transparent.sty' is not loaded}%
    \renewcommand\transparent[1]{}%
  }%
  \providecommand\rotatebox[2]{#2}%
  \newcommand*\fsize{\dimexpr\f@size pt\relax}%
  \newcommand*\lineheight[1]{\fontsize{\fsize}{#1\fsize}\selectfont}%
  \ifx\svgwidth\undefined%
    \setlength{\unitlength}{269.23722875bp}%
    \ifx\svgscale\undefined%
      \relax%
    \else%
      \setlength{\unitlength}{\unitlength * \real{\svgscale}}%
    \fi%
  \else%
    \setlength{\unitlength}{\svgwidth}%
  \fi%
  \global\let\svgwidth\undefined%
  \global\let\svgscale\undefined%
  \makeatother%
  \begin{picture}(1,0.72502441)%
    \lineheight{1}%
    \setlength\tabcolsep{0pt}%
    \put(0,0){\includegraphics[width=\unitlength,page=1]{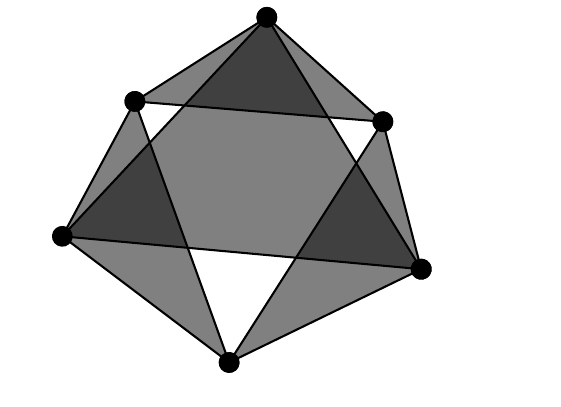}}%
    \put(0.72506948,0.49113131){\color[rgb]{0,0,0}\makebox(0,0)[lt]{\lineheight{0}\smash{\begin{tabular}[t]{l}\(y_{03}\)\end{tabular}}}}%
    \put(0.78900493,0.22384353){\color[rgb]{0,0,0}\makebox(0,0)[lt]{\lineheight{0}\smash{\begin{tabular}[t]{l}\(y_{13}\)\end{tabular}}}}%
    \put(0.12398194,0.53571099){\color[rgb]{0,0,0}\makebox(0,0)[lt]{\lineheight{0}\smash{\begin{tabular}[t]{l}\(y_{02}\)\end{tabular}}}}%
    \put(0.00185713,0.28721352){\color[rgb]{0,0,0}\makebox(0,0)[lt]{\lineheight{0}\smash{\begin{tabular}[t]{l}\(y_{12}\)\end{tabular}}}}%
    \put(0.39015882,0.01280026){\color[rgb]{0,0,0}\makebox(0,0)[lt]{\lineheight{0}\smash{\begin{tabular}[t]{l}\(y_{23}\)\end{tabular}}}}%
    \put(0.51404445,0.68531072){\color[rgb]{0,0,0}\makebox(0,0)[lt]{\lineheight{0}\smash{\begin{tabular}[t]{l}\(y_{01}\)\end{tabular}}}}%
  \end{picture}%
\endgroup%

    \end{subfigure}
    \caption{The complexes \(\D_{1/2}(X,Y)\) (left) and \(\D_{1/2}(Y,X)\) (right) from the proof of Proposition~\ref{prop:counterexample_duality}.}
    \label{fig:counterexample_duality}
\end{figure}


\section{The Dowker-Rips complex as a drop-in replacement for the Dowker complex}\label{sec:dowker_rips_as_drop_in_replacement}

In this section, we present a machine learning application in which using the Dowker-Rips complex instead of the Dowker complex leads to gains in speed while at the same time not negatively impacting performance.
More concretely, it is shown in~\cite{stolz_relational_persistent_homology} that the Dowker complex may be used in a pipeline classifying tumor microenvironments into anti-tumor and pro-tumor macrophage dominant.
We briefly review this pipeline here and refer the reader to~\cite[Section 5.1.1]{stolz_relational_persistent_homology} for details.

First, (an image of) a microenvironment is represented as a two-dimensional point cloud, each point of which is labeled according to whether it represents a blood vessel, necrotic cell, tumor cell or macrophage.
Subsequently, the Dowker complex of one class of points relative to another is constructed; this is done for each of the label combinations macrophage-tumor, tumor-blood vessel and macrophage-blood vessel.
For each of the complexes, persistent homology is computed, represented as a persistence diagram and vectorized in a persistence image, yielding three persistence image for each microenvironment.
These persistence images are flattened into vectors, concatenated and passed to a support vector machine (SVM) for classification of the microenvironment into ``anti-tumor" and ``pro-tumor".
As shown in~\cite{stolz_relational_persistent_homology}, this pipeline achieves a median classification accuracy of 86.6\% across ten runs (controlling for randomized components in the SVM).

We reproduced the above pipeline and result, and subsequently ran the same pipeline with the Dowker complex replaced by the Dowker-Rips complex; see Table~\ref{table:results} for the results.\footnote{
    Python code to run the pipelines is available at \hurl{github.com/m-a-huber/dowker-rips-tumor-prediction}.
    Running it requires our implementations of the Dowker-Rips and the Dowker complex, which are available at \hurl{github.com/m-a-huber/dowker-rips-complex} and \hurl{github.com/m-a-huber/dowker-complex}, respectively.
}
In that table, we report the average classification accuracy with its standard deviation as well as the median accuracy across the ten runs.\footnote{
    The discrepancy between the median accuracy of the pipeline using the Dowker complex reported in~Table~\ref{table:results} and that found in~\cite{stolz_relational_persistent_homology} stems from the fact that we ported the original pipeline from Julia to Python.
}
We thus find that using the Dowker-Rips complex as a drop-in replacement for the Dowker complex in the pipeline above results in essentially the same classification performance.
Crucially, however, we found that computation of the relevant complexes and their persistent homologies was sped up by a factor of over 14 when using the Dowker-Rips complex instead of the Dowker complex.\footnote{
    We ran our experiments on a laptop with a 12th Gen Intel Core i7-1260P processor running at 2.10GHz.
}

\begin{table}
  \caption{Results from microenvironment classification}
  \label{table:results}
  \centering
  \begin{tabular}{lrr}
    \toprule
    Complex used & Mean accuracy & Median accuracy\\
    \midrule
    Dowker-Rips & 86.09±1.39 & 86.05 \\
    Dowker & 85.69±1.49 & 85.51\\
    \bottomrule
  \end{tabular}
\end{table}


\bibliography{biblio}


\end{document}